\documentclass[]{interact}

\usepackage{epstopdf}
\usepackage[caption=false]{subfig}

\usepackage[utf8]{inputenc}
\usepackage{hyperref}
\usepackage{amsmath}

\usepackage{amsfonts} 

\usepackage{url}

\usepackage[numbers,sort&compress]{natbib}
\bibpunct[, ]{[}{]}{,}{n}{,}{,}

\theoremstyle{plain}
\newtheorem{theorem}{Theorem}[section]

\newtheorem{corollary}[theorem]{Corollary}
\newtheorem{proposition}[theorem]{Proposition}

\theoremstyle{definition}

\newtheorem{ass}[theorem]{Assumption}

\theoremstyle{remark}

\def\Box{{\hbox{\raisebox{0.0em}{\rlap{$\sqcap$}}\kern0em%
            \raisebox{-0.0em}{$\sqcup$}}} } 

\def\bepf{\begin{proof}}
\def\epf{\end{proof}}

\def\R{\mathbb{R}}
\def\lbeq#1{\begin{equation} \label{#1}} 
\def\eeq{\end{equation}} 
\def\bary{\begin{array}}
\def\eary{\end{array}}

\def\eps{\varepsilon}

\def\Rz{\mathbb{R}}

\usepackage{array,multirow,graphicx}
\usepackage[utf8]{inputenc}

\usepackage{graphicx}

\usepackage{colortbl}

\usepackage[table]{xcolor}

\usepackage{fancyhdr}

\usepackage{array,multirow,graphicx}

\usepackage{etex}

\usepackage{longtable}
\usepackage{lscape}
\usepackage{hhline}
\usepackage{caption}
\usepackage{blindtext}
\usepackage{lmodern}
\usepackage{textcomp}
\usepackage{chngcntr}
\usepackage{tikz}
\usepackage{csquotes}

\fancypagestyle{plain}{\fancyhf{}\fancyfoot[C]{\bfseries \thepage}}

\newcolumntype{?}{!{\vrule width 1pt}}

\newcounter{subsubsubsection}[subsubsection]
\def\subsubsubsectionmark#1{}

\def\subsubsubsection{\@startsection
	{subsubsubsection}{4}{\z@} {-3.25ex plus -1
		ex minus -.2ex}{1.5ex plus .2ex}{\normalsize\bf}}
\def\l@subsubsubsection{\@dottedtocline{4}{4.8em}
	{4.2em}}

\DeclareMathOperator*{\argmin}{argmin}

\usetikzlibrary{matrix,shapes,arrows,positioning,chains}

\tikzset{
	desicion/.style={
		diamond,
		draw,
		text width=5.5em,
		text badly centered,
		inner sep=0pt
	},
	block/.style={
		rectangle,
		draw,
		text width=6em,
		text centered,
		rounded corners
	},
	cloud/.style={
		draw,
		ellipse,
		minimum height=4em
	},
	descr/.style={
		fill=white,
		inner sep=2.5pt
	},
	connector/.style={
		-latex,
		font=\scriptsize
	},
	rectangle connector/.style={
		connector,
		to path={(\tikztostart) -- ++(#1,0pt) \tikztonodes |- (\tikztotarget) },
		pos=0.5
	},
	rectangle connector/.default=-2cm,
	straight connector/.style={
		connector,
		to path=--(\tikztotarget) \tikztonodes
	}
}
\renewcommand{\Re}{\mathbb{R}}

\begin{document}


\articletype{Full paper}

\title{Worst case complexity bounds for linesearch-type derivative-free algorithms\thanks{Morteza Kimiaei acknowledges the financial support of the Doctoral Program Vienna Graduate School on Computational Optimization (VGSCO) funded by the Austrian Science Foundation under Project No W1260-N35.}}

\author{
\name{A. Brilli\textsuperscript{a}\thanks{A. Brilli. Email: brilli@diag.uniroma1.it}, M. Kimiaei\textsuperscript{b}\thanks{M. Kimiaei. Email: kimiaeim83@univie.ac.at}, G. Liuzzi\textsuperscript{a}\thanks{G. Liuzzi. Email: liuzzi@diag.uniroma1.it} and S. Lucidi\textsuperscript{a}\thanks{S. Lucidi. Email: llucidi@diag.uniroma1.it}}
\affil{\textsuperscript{a}``Sapienza'' Universit\`a di Roma, Dipartimento di Ingegneria Informatica Automatica e Gestionale ``A.Ruberti'', Via Ariosto 25, 00185 Rome, Italy;\\ \textsuperscript{b}Fakult\"at f\"ur Mathematik, Universit\"at Wien, Oskar-Morgenstern-Platz 1, A-1090 Wien, Austria}
}

\maketitle

\begin{abstract}
\nocite{cartis2022evaluation}
This paper is devoted to the analysis of worst case complexity bounds for linesearch-type derivative-free algorithms for the minimization of general non-convex smooth functions. We prove that two linesearch-type algorithms enjoy the same complexity properties which have been proved for  pattern and direct search algorithms. In particular, we consider two derivative-free algorithms based on two different linesearch techniques and manage to prove that the number of iterations and of function evaluations required to drive the norm of the gradient of the objective function below a given threshold $\epsilon$ is ${\cal O}(\epsilon^{-2})$ in the worst case. 
\end{abstract}


\begin{amscode} 90C56
\end{amscode}

\section{Introduction}
In this paper we consider the following unconstrained minimization problem 
\begin{equation}\tag{P}\label{uncprob}
\min_{x\in\Re^n}\ f(x)
\end{equation}
We assume that $f:\Re^n\to\Re$ is a black-box function which is known by means of an oracle that only outputs function values. Hence, derivatives of $f$ cannot neither be approximated nor computed explicitly. Even though derivatives of $f$ are not available, we assume that $f$ is continuously differentiable with Lipschitz-continuous gradient.  

Over the past decade, the analysis of worst case complexity for optimization algorithms has gained more and more interest and attracted many researchers \cite{cartis2022evaluation,amaral2022complexity}. Specifically for derivative-free algorithms, 
in \cite{vicente:13,Dodangeh:16} worst case complexity bounds have derived for direct search methods using sufficient decrease in $f$. In particular, it has been proved that direct search methods (based on a search step and a poll step and using sufficient decrease acceptability) require at most ${\cal O}(\epsilon^{-2})$ iterations and ${\cal O}(n^2\epsilon^{-2})$ function evaluations to drive the norm of the gradient below $\epsilon$. 

Analogous results for linesearch-based derivative algorithms (see e.g. \cite{lucidi2002global,fasano2014linesearch}) have not yet been established. The latter algorithms typically have stronger asymptotic convergence properties which are tied to the use of suitable though more complex extrapolation techniques. Hence, the analysis of the worst case complexity of such algorithms is more complicated.  However, we manage to obtain almost the same complexity bounds for both the number of iterations and function evaluations as those proved for direct search methods.

The paper is organized as follows. In section \ref{sec:algo} we describe a general framework of derivative-free algorithm based on linesearch techniques and we propose two variants, namely a ``standard'' and a ``new'' version of the framework. In section \ref{sec:asymp} we first derive a bound on the norm of the gradient and then we prove the asymptotic convergence of the algorithm model to stationary points of the objective function. In section \ref{sec:complexity} we derive the worst case complexity bounds for the number of iterations and function evaluations required by the algorithm to drive the norm of the gradient below a prefixed tolerance. Finally, in section \ref{sec:conclusion} we draw some conclusions.


\section{Linesearch-type algorithms}\label{sec:algo}
In this section we introduce two linesearch-type derivative-free algorithms. They both share the same general framework which is reported below.


\par\noindent\medskip

\noindent\framebox[\textwidth]{\parbox{0.95\textwidth}{
\par\bigskip
\centerline{ {\bf  Linesearch Algorithm Model (LAM)}}
\par\medskip
  {\bf Data.} $c\in(0,1),\theta\in(0,1)$, $x_0\in \R^n$, $\tilde\alpha_0^i > 0$, $i\in \{1,\dots,n\}$, and set $d_0^i=e^i$, for $i=1,\ldots,n$.

\par\medskip

{\bf For} $k=0,1,\dots$


\qquad Set $y_k^1=x_k$.

\qquad {\bf For} $i=1,\dots,n$

%

\qquad\qquad Let $\bar\alpha_k^i = \max\{\tilde\alpha_k^i,c\displaystyle\max_{j=1,\dots,n}\{\tilde\alpha_k^j\}\}.$


\qquad\qquad Compute $\alpha$ and $d$ by the {\tt DF-Linesearch}$(\bar\alpha_k^i,y_k^i,d_k^i;\alpha,d).$

\medskip

\qquad\qquad {\bf If} $\alpha = 0$ {\bf then} set $\alpha_k^i = 0$ and $\tilde\alpha_{k+1}^i = \theta\bar\alpha_k^i$.

\qquad\qquad {\bf else} set $\alpha_k^i = \alpha$, $\tilde\alpha_{k+1}^i = \alpha$.



\medskip

 \qquad\qquad Set $d_{k+1}^i=d$, $y_k^{i+1}=y_k^i+\alpha_k^id_{k+1}^i$.
\par\medskip

 \qquad {\bf End For}

 \qquad Set $x_{k+1}=y_k^{n+1}$.

{\bf End For}

\par\bigskip\noindent

}}

\par\medskip\medskip

As we can see, at each iteration $k$, LAM performs an exploration of the space around the current iterate $x_k$ using the coordinate directions and producing the points $y_k^i$, $i=1,\dots,n+1$ (note that $y_k^1 = x_k$).

More in particular, points $y_k^i$, for $i=2,\dots,n+1$, are computed by means of a suitable derivative-free linesearch, namely the {\tt DF-Linesearch} procedure. The linesearch is invoked by passing a tentative step size, i.e. 
\[
\bar\alpha_k^i = \max\{\tilde\alpha_k^i,c\displaystyle\max_{j=1,\dots,n}\{\tilde\alpha_k^j\}\}.
\] 
Then, an actual step size, i.e. $\alpha$ is produced, which can either be $0$ or strictly greater than zero. When $\alpha = 0$, the linesearch ``fails'' and the tentaive step for the next iteration is reduced; when $\alpha > 0$, the tentative step for the next iteration is updated and possibly augmented.  

Concerning the actual definition of suitable linesearch procedures, i.e. of the {\tt DF-Linesearch} procedure, we report two possible and rather different schemes. The first scheme (the standard DF-Linesearch) is a quite standard linesearch based on an extrapolation with sufficient decrease (see for example \cite{lucidi2002global,fasano2014linesearch}). Roughly speaking, if sufficient decrease can be obtained with the initial step size along $d$ or $-d$ an extrapolation is used to exploit as much as possible the ``descent'' property of the search direction until sufficient decrease with respect to the initial point can be achieved, i.e.
\[
\displaystyle f \bigg (y+\frac{\alpha} {\delta} \hat d \bigg )\leq f(y)-\gamma\frac{\alpha^2} {\delta^2}.
\]

\par\medskip\medskip
\noindent\framebox[\textwidth]{\parbox{0.95\textwidth}{
\par\bigskip
\centerline{{\bf (standard) DF-Linesearch  ($\bar\alpha,y,d;\alpha,\hat d$).}}
\begin{itemize}
 \item[ ] {\bf Data.} $\gamma > 0$, $\delta \in (0,1)$.
 \item[ ] {\bf Step 1.} Set $\alpha=\bar\alpha$, $\hat d \leftarrow d$.
 \item[ ] {\bf Step 2.} {\bf If} $f(y+\alpha \hat d)\le f(y)-\gamma\alpha^2$ {\bf then} go to Step 5.
 \item[ ] {\bf Step 3.} {\bf If} $f(y-\alpha \hat d)\le f(y)-\gamma\alpha^2$ {\bf then} set $\hat d \leftarrow -d$ and go to Step 5.
 \item[ ] {\bf Step 4.} Set $\alpha = 0$ and {\bf return}.
 \item[ ] {\bf Step 5.} {\bf While}   $\displaystyle f \bigg (y+\frac{\alpha} {\delta} \hat d \bigg )\leq f(y)-\gamma\left(\frac{\alpha} {\delta}\right)^2$ 
 \item[ ] \qquad\qquad\qquad\quad $\alpha \leftarrow {\alpha}/{\delta}$.
 \item[ ] {\bf Step 6.} {\bf Return} $(\alpha, \hat d)$
\end{itemize}
\par\bigskip\noindent
}}

One possible disadvantage of the standard DF-Linesearch is that intermediate points (those produced during the extrapolation) do not play a very significant role. In particular, it could happen that the final step size is not the one producing the best reduction with respect to the initial point. Driven by the preceding consideration and inspired by \cite{grippo1988global}, we try to take into bigger account intermediate points by introducing a somewhat new DF-Linesearch procedure which is based on a different sufficient decrease criterion. In particular, as we can see, the sufficient decrease condition is checked between consecutive points, i.e.
\[
\displaystyle f \bigg (y+\frac{\alpha}{\delta} \hat d \bigg )\leq f(y+\alpha \hat{d})-\gamma\bigg(\bigg(\frac{1}{\delta}-1\bigg)\alpha\bigg)^2.
\]

\par\medskip
\noindent\framebox[\textwidth]{\parbox{0.95\textwidth}{
\par\bigskip
\centerline{{\bf (new) DF-Linesearch  ($\bar\alpha,y,d;\alpha,\hat d$).}}
\begin{itemize}
 \item[ ] {\bf Data.} $\gamma > 0$, $\delta \in (0,1)$.
 \item[ ] {\bf Step 1.} Set $\alpha=\bar\alpha$, $\hat d \leftarrow d$.
 \item[ ] {\bf Step 2.} {\bf If} $f(y+\alpha \hat d)\le f(y)-\gamma\alpha^2$ {\bf then} go to Step 5.
 \item[ ] {\bf Step 3.} {\bf If} $f(y-\alpha \hat d)\le f(y)-\gamma\alpha^2$ {\bf then} set $\hat d \leftarrow -d$ and go to Step 5.
 \item[ ] {\bf Step 4.} Set $\alpha = 0$ and {\bf return}.
 \item[ ] {\bf Step 5.} {\bf While}   $\displaystyle f \bigg (y+\frac{\alpha}{\delta} \hat d \bigg )\leq f(y+\alpha \hat{d})-\gamma\bigg(\bigg(\frac{1}{\delta}-1\bigg)\alpha\bigg)^2$ 
 \item[ ] \qquad\qquad\qquad\quad $\alpha \leftarrow {\alpha}/{\delta}$.
 \item[ ] {\bf Step 6.} {\bf Return} $(\alpha, \hat d)$
\end{itemize}
\par\bigskip\noindent
}}

\par\bigskip\bigskip

Note that the new linesearch technique has a different behavior than the classical one as it can be seen in Figures \ref{fig:my_label1} and \ref{fig:my_label2}. In particular, Figure \ref{fig:my_label1} shows that the new lineasearch could be less restrictive than the classical one, i.e. it accepts steps that would be refused by the classical one. Furthermore, Figure \ref{fig:my_label2} shows that  steps producing moderate reduction with respect to the last accepted point are not accepted by the new linesearch whereas they would have been accepted by the classical one.


\begin{figure}[htb]
    \centering
    \includegraphics[width=0.8\textwidth]{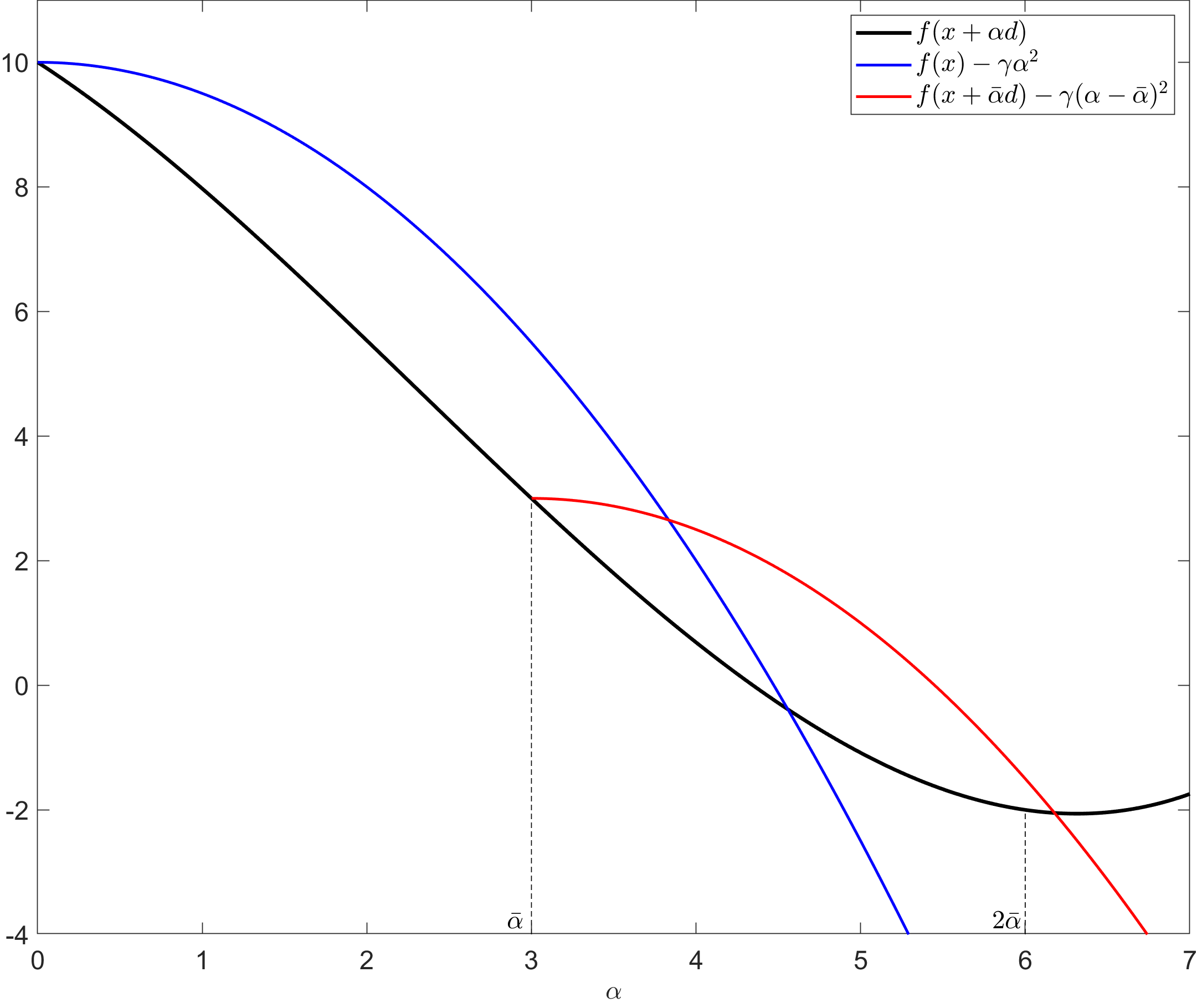}
    \caption{Comparison between new and classical linesearch. The function value along the line $x+\alpha d$ is reported in black; the blue line represents the sufficient reduction value in the classical linesearch (i.e. $f(x)-\gamma\alpha^2$) whereas the red line depicts the sufficient reduction of the new linesearch (i.e. $f(x+\bar\alpha d) -\gamma(\alpha-\bar\alpha)^2$). Here, $\bar\alpha = 3$ and $2\bar\alpha = 6$.}
    \label{fig:my_label1}
\end{figure}
\begin{figure}[htb]
    \centering
    \includegraphics[width=0.8\textwidth]{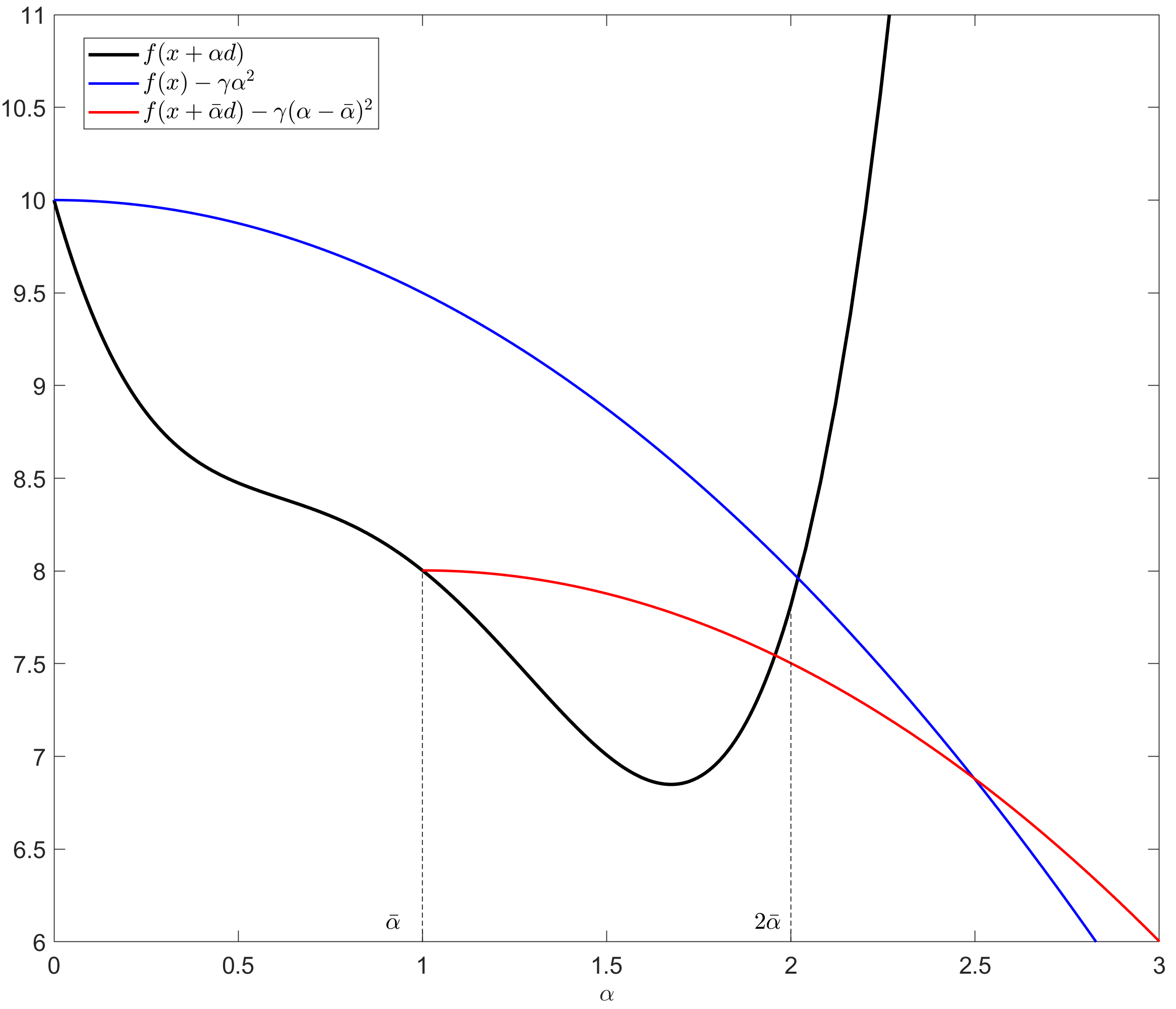}
    \caption{Comparison between new and classical linesearch. The function value along the line $x+\alpha d$ is reported in black; the blue line represents the sufficient reduction value in the classical linesearch (i.e. $f(x)-\gamma\alpha^2$) whereas the red line depicts the sufficient reduction of the new linesearch (i.e. $f(x+\bar\alpha d) -\gamma(\alpha-\bar\alpha)^2$). Here, $\bar\alpha = 1$ and $2\bar\alpha = 2$.}
    \label{fig:my_label2}
\end{figure}


\section{Asymptotic convergence analysis for {\rm LAM}}\label{sec:asymp}

In order to carry out the convergence analysis for LAM, we need the following standard assumption: 

\begin{ass}\label{gradlip}
The function $f$ is continuously differentiable on $\Rz^n$, 
and its gradient is Lipschitz continuous with Lipschitz constant $L$, i.e. for all $x,y\in\Re^n$,
\[
\|\nabla f(x)-\nabla f(y)\| \leq L \|x-y\|.
\]
\end{ass}

In the following, the standard LAM employs the standard DF-Linesearch and the new LAM employs the new DF-Linesearch.

First of all, we derive an upper bound on the norm of $\nabla f$ at each iteration $k$.

\begin{proposition}\label{bound_grad}
	Suppose that Assumption \ref{gradlip} holds and that $f$ is bounded from below.  Let $\{x_k\}$ be the sequence produced by the {LAM} framework.
	Then, for each $k$ such that $x_{k+1}\neq  x_k$
	\[
	\|\nabla f(x_k)\| \leq \sqrt{n}\left(\frac{\gamma+L(\sqrt{n}+1)}{\min\{\theta,\delta\}} \right)\max_{i=1,\dots,n}\{\tilde\alpha^i_{k+1}\},
	\]
	whereas, for each $k$ such that $x_{k+1} = x_k$
	\[
	\|\nabla f(x_k)\| \leq \sqrt{n}\frac{\gamma + L}{\theta}\max_{i=1,\dots,n}\{\tilde\alpha^i_{k+1}\}.
	\]
\end{proposition}


%
\begin{proof}
For each iteration $k$ such that $x_{k+1}\neq x_{k}$ and every index $i=1,\dots,n$, one of two cases can occur:

\underline{Case (i)}. By $\alpha_{k}^i = 0$, and $\tilde\alpha_{k+1}^i = \theta\bar\alpha_{k}^i$, we have:
\begin{eqnarray*}
 && f(y_{k}^i+\bar\alpha_k^i e_i)> f(y_{k}^i)-\gamma(\bar\alpha_k^i)^2,\\
 && f(y_{k}^i-\bar\alpha_k^i e_i)> f(y_{k}^i)-\gamma(\bar\alpha_k^i)^2.
\end{eqnarray*}
Then we get from the Mean Value Theorem
\begin{eqnarray}
&& \nabla f(u_k^i )^Te_i > -\gamma\bar\alpha_k^i,\label{eq1}\\
&& \nabla f(v_k^i )^Te_i <  \gamma\bar\alpha_k^i\label{eq2}
\end{eqnarray}

where $u_k^i=y_{k}^i+\lambda_k^i \bar\alpha_k^i e_i$ and $v_k^i=y_{k}^i-\mu_k^i\bar\alpha_k^i e_i$ with $\lambda_k^i, \mu_k^i \in (0,1)$.
{}From (\ref{eq1}) and (\ref{eq2}) and the Lipschitz continuity of $\nabla f$, we have that
\begin{eqnarray*}
 && \nabla f(x_k)^T e_i >          - \gamma\bar\alpha_k^i -L \|x_k - u_k^i\| >          - \gamma\bar\alpha_k^i -L \|x_k - y_{k}^i\| - L \bar\alpha_k^i,\\
 && \nabla f(x_k)^T e_i < \phantom{-}\gamma\bar\alpha_k^i +L \|x_k - v_k^i\| < \phantom{-}\gamma\bar\alpha_k^i +L \|x_k - y_{k}^i\| + L \bar\alpha_k^i.
\end{eqnarray*}
Hence
\[
\begin{split}
 & |\nabla f(x_k)^T e_i| < (\gamma+L)\bar\alpha_k^i +L \|x_k - y_{k}^i\| \leq (\gamma+L)\bar\alpha_k^i +L\sqrt{n}\max_{i=1,\dots,n}\{\tilde\alpha_{{k+1}}^i\} = \\ &\qquad (\gamma+L)\frac{\tilde\alpha_{k+1}^i}{\theta} +L\sqrt{n}\max_{i=1,\dots,n}\{\tilde\alpha_{{k+1}}^i\} \leq \frac{(\gamma+L)}{\theta}\max_{i=1,\dots,n}\{\tilde\alpha_{k+1}^i\} +L\sqrt{n}\max_{i=1,\dots,n}\{\tilde\alpha_{{k+1}}^i\},
 \end{split}
\]
so that
\[
 |\nabla f(x_k)^T e_i| \leq \left(\frac{\gamma+L(\sqrt{n}+1)}{\theta} \right)\max_{i=1,\dots,n}\{\tilde\alpha_{k+1}^i\}.
\]

\underline{Case (ii) (for the standard LAM)}. From $\alpha_{k}^i = \alpha$, and $\tilde\alpha_{k+1}^i = \alpha\geq \bar\alpha_k^i$, we result in either
\[
  f\left(y_{k}^i+\frac{\tilde\alpha_{k+1}^i}{\delta} e_i\right)> f(y_{k}^i)-\gamma\left(\frac{\tilde\alpha_{k+1}^i}{\delta}\right)^2,\quad
  f(y_{k}^i) \geq f(y_{k}^i + \tilde\alpha_{k+1}^i e_i) + \gamma(\tilde\alpha_{k+1}^i)^2
\]
or
\[
  f\left(y_{k}^i-\frac{\tilde\alpha_{k+1}^i}{\delta} e_i\right)> f(y_{k}^i)-\gamma\left(\frac{\tilde\alpha_{k+1}^i}{\delta}\right)^2,\quad
  f(y_{k}^i) \geq f(y_{k}^i - \tilde\alpha_{k+1}^i e_i) + \gamma(\tilde\alpha_{k+1}^i)^2.
\]
Then, we get,
\begin{equation}\label{caseii_cond1_1}
\nabla f(\bar u_k^i )^Te_i > -\gamma\frac{\tilde\alpha_{k+1}^i}{\delta}, \quad -\nabla f(\hat u_k^i )^Te_i \geq \gamma(\tilde\alpha_{k+1}^i),
\end{equation}
or
\begin{equation}\label{caseii_cond2_1}
\nabla f(\bar v_k^i )^Te_i <  \gamma\frac{\tilde\alpha_{k+1}^i}{\delta}, \quad -\nabla f(\hat v_k^i )^Te_i \leq -\gamma(\tilde\alpha_{k+1}^i)
\end{equation}
where 
$\bar u_k^i=y_{k}^i+\bar\lambda_k^i \displaystyle\frac{\tilde\alpha_{k+1}^i}{\delta} e_i$,
$\hat u_k^i=y_{k}^i+\hat\lambda_k^i \displaystyle\frac{\tilde\alpha_{k+1}^i}{\delta} e_i$, 
$\bar v_k^i=y_{k}^i-\bar\mu_k^i\displaystyle\frac{\tilde\alpha_{k+1}^i}{\delta} e_i$, and
$\hat v_k^i=y_{k}^i-\hat\mu_k^i\displaystyle\frac{\tilde\alpha_{k+1}^i}{\delta} e_i$, with 
$\bar\lambda_k^i, \hat\lambda_k^i, \bar\mu_k^i, \hat\mu_k^i \in (0,1)$.

{}When (\ref{caseii_cond1_1}) holds, from $\nabla f(\bar u_k^i )^Te_i > -\gamma\displaystyle\frac{\tilde\alpha_{k+1}^i}{\delta}$ we can write
\[
 [\nabla f(\bar u_k^i ) - \nabla f(x_k) + \nabla f(x_k)]^Te_i > -\gamma\frac{\tilde\alpha_{k+1}^i}{\delta},
\]
so that we obtain
\begin{equation}\label{bound1_left_1}
 \nabla f(x_k)^Te_i > -\gamma\frac{\tilde\alpha_{k+1}^i}{\delta} - L\|x_k-\bar u_k^i\| > -\gamma\frac{\tilde\alpha_{k+1}^i}{\delta} -L \|x_k - y_{k}^i\| - L \frac{\tilde\alpha_{k+1}^i}{\delta}.
\end{equation}
{}From $\nabla f(\hat u_k^i )^Te_i \leq -\gamma\tilde\alpha_{k+1}^i$ in (\ref{caseii_cond1_1}) we can write
\[
 [\nabla f(\hat u_k^i ) - \nabla f(x_k) + \nabla f(x_k)]^Te_i \leq -\gamma\tilde\alpha_{k+1}^i,
\]
so that, in this case, we obtain
\begin{equation}\label{bound1_right_1}
 \nabla f(x_k)^Te_i \leq  -\gamma\tilde\alpha_{k+1}^i + L\|x_k-\hat u_k^i\| \leq  \gamma\frac{\tilde\alpha_{k+1}^i}{\delta} + L \|x_k - y_{k}^i\| + L \frac{\tilde\alpha_{k+1}^i}{\delta}.
\end{equation}

\underline{Case (ii) (for the new LAM)}. From $\alpha_{k}^i = \alpha$, and $\tilde\alpha_{k+1}^i = \alpha\geq \bar\alpha_k^i$, we results in either
\[
\begin{split}
  & f\left(y_{k}^i+\frac{\tilde\alpha_{k+1}^i}{\delta} e_i\right)> f(y_{k}^i+\tilde\alpha_{k+1}^i e_i)-\gamma\left(\frac{1}{\delta}-1\right)^2(\tilde\alpha_{k+1}^i)^2,\\ 
  & f(y_{k}^i+\delta \tilde\alpha_{k+1}^i e_i) \geq f(y_{k}^i + \tilde\alpha_{k+1}^i e_i) + \gamma(1-\delta)^2(\tilde\alpha_{k+1}^i)^2
\end{split}
\]
or
\[
\begin{split}
  & f\left(y_{k}^i-\frac{\tilde\alpha_{k+1}^i}{\delta} e_i\right)> f(y_{k}^i-\tilde\alpha_{k+1}^i e_i)-\gamma\left(\frac{1}{\delta}-1\right)^2(\tilde\alpha_{k+1}^i)^2,\\ 
  & f(y_{k}^i-\delta \tilde\alpha_{k+1}^i e_i) \geq f(y_{k}^i - \tilde\alpha_{k+1}^i e_i) + \gamma(1-\delta)^2(\tilde\alpha_{k+1}^i)^2.
\end{split}
\]
Then, we get,
\begin{equation}\label{caseii_cond1}
\nabla f(\bar u_k^i )^Te_i > -\gamma\left(\frac{1-\delta}{\delta}\right)\tilde\alpha_{k+1}^i, \quad -\nabla f(\hat u_k^i )^Te_i \geq \gamma(1-\delta)\tilde\alpha_{k+1}^i,
\end{equation}
or
\begin{equation}\label{caseii_cond2}
\nabla f(\bar v_k^i )^Te_i <  \gamma\left(\frac{1-\delta}{\delta}\right)\tilde\alpha_{k+1}^i, \quad -\nabla f(\hat v_k^i )^Te_i \leq -\gamma(1-\delta)\tilde\alpha_{k+1}^i
\end{equation}
{where 
$\bar u_k^i=y_{k}^i+\bar\lambda_k^i \left(\dfrac{1-\delta}{\delta}\right)\tilde\alpha_{k+1}^i e_i$,
$\hat u_k^i=y_{k}^i+\hat\lambda_k^i (1-\delta)\tilde\alpha_{k+1}^i e_i$, 
$\bar v_k^i=y_{k}^i-\bar\mu_k^i\left(\dfrac{1-\delta}{\delta}\right)\tilde\alpha_{k+1}^i e_i$, and
$\hat v_k^i=y_{k}^i-\hat\mu_k^i(1-\delta)\tilde\alpha_{k+1}^i e_i$, with 
$\bar\lambda_k^i, \hat\lambda_k^i, \bar\mu_k^i, \hat\mu_k^i \in (0,1)$.}

{}When (\ref{caseii_cond1}) holds, from $\nabla f(\bar u_k^i )^Te_i > -\gamma\left(\dfrac{1-\delta}{\delta}\right)\tilde\alpha_{k+1}^i$ we can write
\[
 [\nabla f(\bar u_k^i ) - \nabla f(x_k) + \nabla f(x_k)]^Te_i > -\gamma\left(\frac{1-\delta}{\delta}\right)\tilde\alpha_{k+1}^i,
\]
so that we obtain
\begin{equation}\label{bound1_left}
\begin{split}
 \nabla f(x_k)^Te_i & > -\gamma\left(\frac{1-\delta}{\delta}\right)\tilde\alpha_{k+1}^i - L\|x_k-\bar u_k^i\| \\
 & > -\gamma\left(\frac{1-\delta}{\delta}\right)\tilde\alpha_{k+1}^i -L \|x_k - y_{k}^i\| - L (\frac{1-\delta}{\delta})\tilde\alpha_{k+1}^i.
 \end{split}
\end{equation}
{}From $\nabla f(\hat u_k^i )^Te_i \leq -\gamma(1-\delta)\tilde\alpha_{k+1}^i$ in (\ref{caseii_cond1}), we can write
\[
 [\nabla f(\hat u_k^i ) - \nabla f(x_k) + \nabla f(x_k)]^Te_i \leq -\gamma(1-\delta)\tilde\alpha_{k+1}^i,
\]
so that, in this case, we obtain
\begin{equation}\label{bound1_right}
\begin{split}
 \nabla f(x_k)^Te_i & \leq  -\gamma(1-\delta)\tilde\alpha_{k+1}^i + L\|x_k-\hat u_k^i\| \\
 & \leq  \gamma\left(\frac{1-\delta}{\delta}\right)\tilde\alpha_{k+1}^i + L \|x_k - y_{k}^i\| + L \left(\frac{1-\delta}{\delta}\right)\tilde\alpha_{k+1}^i.
\end{split}
\end{equation}

Now, considering (\ref{bound1_left_1}) and (\ref{bound1_right_1}) for the standard LAM and (\ref{bound1_left}) and (\ref{bound1_right}) for the new LAM, we get
\[
 |\nabla f(x_k)^T e_i| \leq \left(\frac{\gamma+L(\sqrt{n}+1)}{\delta} \right)\max_{i=1,\dots,n}\{\tilde\alpha_{k+1}^i\}.
\]
The same bound can be obtained when (\ref{caseii_cond2}) holds. Thus, 
finally, we obtain 
\[
 \|\nabla f(x_k)\| \leq\sqrt{n}\left(\frac{\gamma+L(\sqrt{n}+1)}{\min\{\theta,\delta\}} \right)\max_{i=1,\dots,n}\{\tilde\alpha_{k+1}^i\}.
\]
On the other hand, for each iteration $k$ such that $x_{k+1}=x_{k}$, i.e. $y_{k}^i = x_{k}$ for all $i=1,\dots,n+1$, we have for every index $i=1,\dots,n$
\begin{eqnarray*}
	&& f\left(x_{k}+\frac{\tilde\alpha_{k+1}^i}{\theta} e_i\right)> f(x_{k})-\gamma\left(\frac{\tilde\alpha_{k+1}^i}{\theta}\right)^2,\\
	&& f\left(x_{k}-\frac{\tilde\alpha_{k+1}^i}{\theta} e_i\right)> f(x_{k})-\gamma\left(\frac{\tilde\alpha_{k+1}^i}{\theta}\right)^2.
\end{eqnarray*}
Then we get from the Mean Value Theorem
\begin{eqnarray}
	&& \nabla f(u_k^i )^Te_i > -\gamma\frac{\tilde\alpha_{k+1}^i}{\theta},\label{eq1_1}\\
	&& \nabla f(v_k^i )^Te_i <  \gamma\frac{\tilde\alpha_{k+1}^i}{\theta}\label{eq2_1},
\end{eqnarray}
where $u_k^i=x_{k}+\lambda_k^i \dfrac{\tilde\alpha_{k+1}^i}{\theta} e_i$ and $v_k^i=x_{k}-\mu_k^i\dfrac{\tilde\alpha_{k+1}^i}{\theta} e_i$ with $\lambda_k^i, \mu_k^i \in (0,1)$.
{}From (\ref{eq1_1}) and (\ref{eq2_1}) and the Lipschitz continuity of $\nabla f$, we have that
\begin{eqnarray*}
	&& \nabla f(x_k)^T e_i >          - \gamma\frac{\tilde\alpha_{k+1}^i}{\theta} -L \|x_k - u_k^i\| >          - \gamma\frac{\tilde\alpha_{k+1}^i}{\theta}  - L \frac{\tilde\alpha_{k+1}^i}{\theta},\\
	&& \nabla f(x_k)^T e_i < \phantom{-}\gamma\frac{\tilde\alpha_{k+1}^i}{\theta} +L \|x_k - v_k^i\| < \phantom{-}\gamma\frac{\tilde\alpha_{k+1}^i}{\theta}  + L \frac{\tilde\alpha_{k+1}^i}{\theta}.
\end{eqnarray*}
Hence
\[
|\nabla f(x_k)^T e_i| < (\gamma+L)\frac{\tilde\alpha_{k+1}^i}{\theta},  
\]
so that we can write
\[
\|\nabla f(x_k)\| \leq \sqrt{n}\frac{\gamma + L}{\theta}\max_{i=1,\dots,n}\{\tilde\alpha_{k+1}^i\}
\]
concluding the proof.
\end{proof}

Now, we prove that the sequences of initial stepsizes $\tilde\alpha_k^i$, $i=1,\dots,n$, are all convergent to zero.

\begin{proposition}\label{prop:tildealphatozero}
Assume that $f$ is bounded from below. Then, the LAM framework produces sequences $\{\tilde\alpha_k^i\}$, $i=1,\dots,n$, such that
\[
\lim_{k\to\infty}\max_{i=1,\dots,n}\{\tilde\alpha_k^i\} = 0
\]
\end{proposition}

\begin{proof}
For each index $i$, we split the set of iteration indices $\{0,1,2,\dots\}$ into the two subsets $K_1$ and $K_2$, namely
\begin{itemize}
\item[i)] $k\in K_1$ when $\alpha_k^i = 0$, $\tilde\alpha_{k+1}^i = \theta\bar\alpha_k^i$;
\item[ii)] $k\in K_2$ when $\alpha_k^i\neq 0$, $\tilde\alpha_{k+1}^i = \alpha_k^i \geq \bar\alpha_k^i$.
\end{itemize}
Note that the sets $K_1$ and $K_2$ cannot be both finite. When $k\in K_2$, by the DF-Linesearch procedure definition, we have,
for the standard LAM 
\begin{equation}\label{suffdeccond_1}
f(x_{k+1})  \leq f(y_k^i + \alpha_k^id_{k+1}^i) \leq f(y_k^i) - \gamma(\alpha_k^i)^2 \leq f(x_k)- \gamma(\alpha_k^i)^2
\end{equation}
and for the new LAM\begin{equation}\label{suffdeccond}
\begin{split}
f(x_{k+1}) &\leq f(y_k^i + \alpha_k^id_{k+1}^i) \\
    &\leq f(y_k^i+\delta\alpha_k^i d_{k+1}^i) - \gamma(1-\delta)^2(\alpha_k^i)^2 \leq f(x_k)- \gamma(1-\delta)^2(\alpha_k^i)^2.
\end{split}
\end{equation}
Taking into account the boundedness assumption on $f$, it follows from (\ref{suffdeccond_1}) and (\ref{suffdeccond}) that $\{f(x_k)\}$ tends to a limit $\bar f$ and we obtain that
\begin{eqnarray}
&& \lim_{k\to\infty,k\in K_2} \alpha_k^i = 0, \label{alphaK2}\\
&& \lim_{k\to\infty,k\in K_2} \bar\alpha_k^i = 0.\label{baralphaK2} 
\end{eqnarray}
Then, since for all $k\in K_1$, $\alpha_k^i = 0$ by definition, we have that
\begin{equation}\label{allalphatozero}
\lim_{k\to\infty} \alpha_k^i = 0.
\end{equation}
Now, as concerns the $\tilde\alpha_k^i$'s, if $K_2$ is an infinite subset, from (\ref{allalphatozero}) and the fact that $\tilde\alpha_{k+1}^i = \alpha_k^i$, we obtain
\begin{equation}\label{tildealphaK2}
\lim_{k\to\infty,k\in K_2} \tilde\alpha_{k+1}^i = 0.
\end{equation}
On the other hand, let us suppose that $K_1$ is infinite. For each $k\in K_1$, let $m_k$ be the biggest index such that $m_k < k$ and $m_k\in K_2$. Then, we have
\begin{equation}\label{relmk}
 \tilde\alpha_{k+1}^i = \theta^{k-m_k+1}\bar\alpha_{m_k}^i
\end{equation}
and we assume that $m_k = 0$ when $K_2 = \emptyset$. Now, as $k\to\infty$, $k\in K_1$, either $K_2$ is infinite too, implying that $m_k\to\infty$, or $K_2$ is finite, implying that $k-m_k+1\to\infty$. Hence, if $K_1$ is infinite,  (\ref{relmk}) along with (\ref{baralphaK2}) or the fact that $\theta\in (0,1)$, yields that
\begin{equation}\label{tildealphaK1}
\lim_{k\to\infty,k\in K_1}\tilde\alpha_{k+1}^i = 0.
\end{equation}
The proof is concluded recalling   (\ref{tildealphaK2}) and (\ref{tildealphaK1}).
\end{proof}

\begin{corollary}\label{corollary:convergence}
Suppose that Assumption \ref{gradlip} holds and that $f$ is bounded from below. Then, LAM produces an infinite sequence $\{x_k\}$ such that
\[
\lim_{k\to\infty} \|\nabla f(x_k)\| = 0.
\]
\end{corollary}
\begin{proof}
The proof easily follows recalling Proposition \ref{bound_grad} and Proposition \ref{prop:tildealphatozero}.
\end{proof}

{\bf Remark} {\em Note that the result of Proposition \ref{prop:tildealphatozero} is somewhat stronger than analogous results for GPS \cite{torczon:97} and MADS-type \cite{audet:06} algorithms. Indeed, for those algorithms it is only possible to show that a subsequence of the stepsizes converges to zero. As a consequence, also the result of Corollary \ref{corollary:convergence} is stronger in that it states that every limit point of the sequence of iterates is stationary.}

\section{Complexity bound for {\rm LAM} in the non-convex case}\label{sec:complexity}

\begin{proposition}
 Suppose that Assumption \ref{gradlip} holds and that $f$ is bounded from below. Given any $\epsilon\in(0,1)$, assume that $\bar\jmath+1$ is the first iteration such that $\|\nabla f(x_{\bar\jmath})\|\leq\epsilon$, i.e. 
 $\|\nabla f(x_k)\| > \epsilon$, for all $k=0,1,\dots,\bar\jmath$. Then,
\begin{equation}\label{bound_on_iterations}
   \bar\jmath \leq \left\lceil\frac{n c_1^2 (\Phi_0 - f_{\min})}{\tilde c}\eps^{-2}\right\rceil = {\cal O}(\epsilon^{-2}),
\end{equation}
where
\begin{eqnarray}
 c_1 &=&\max \left\{ \left(\displaystyle\frac{\gamma+L(\sqrt{n}+1)}{\min\{\theta,\delta\}} \right), \ \left(\displaystyle\frac{\gamma+L}{\theta} \right)
 \right \},\label{c1def}\\
 \tilde c & = & \min\left\{\gamma c^2,\gamma\left(1-\displaystyle\frac{c^2}{2}\right)\right\}.\label{ctildedef}
\end{eqnarray}
\end{proposition}

\begin{proof}
For all iterations $k=0,1,\dots,\bar\jmath$, by Proposition \ref{bound_grad}, we have that
\begin{equation}\label{bound_maxalpha}
 \epsilon < \|\nabla f(x_k)\| \leq c_1\sqrt{n}\max_{i=1,\dots,n}\{\tilde\alpha^i_{k+1}\}.
\end{equation}
Hence,
\begin{equation}\label{bound_maxalpha1}
 \max_{i=1,\dots,n}\{\tilde\alpha^i_{k+1}\} \geq \frac{\epsilon}{c_1\sqrt{n}},
\end{equation}

Let us introduce, 
\[
\Phi_k = f(x_k) + \frac{1}{2}c^2\gamma \max_{i=1,\dots,n}\{\tilde\alpha_k^i\}^2
\]
and note that, since $f$ is bounded from below,
\begin{equation}\label{eq:lowerbound_Phi}
\Phi_k \geq f_{\min}.
\end{equation}
If the $(k-1)$-th iteration is of failure, then
\begin{itemize}
	\item[i)] $f(x_{k}) = f(x_{k-1})$ and
	\item[ii)] $\tilde\alpha_k^i = \theta\bar\alpha_{k-1}^i$, for every $i=1,\dots,n$.
\end{itemize}
In particular, since, for every $i=1,\dots,n$, $\bar\alpha_{k-1}^i = \max\{\tilde\alpha_{k-1}^i,c\max\{\tilde\alpha_{k-1}^i\}\}$, we have either
\[
   \tilde\alpha_k^i = \theta\tilde\alpha_{k-1}^i \leq \theta  \max\{\tilde\alpha_{k-1}^i\}
\] 
or
\[
  \tilde\alpha_k^i = \theta c \max\{\tilde\alpha_{k-1}^i\} \leq \theta  \max\{\tilde\alpha_{k-1}^i\}.
\]
Hence, in case of failure, we can write $\tilde\alpha_k^i \leq \theta  \max\{\tilde\alpha_{k-1}^i\}$, so that \\ $\max\{\tilde\alpha_k^i\}\leq\theta  \max\{\tilde\alpha_{k-1}^i\}$. Then, from
\[
\Phi_k - \Phi_{k-1} = f(x_k)  - f(x_{k-1}) + \frac{1}{2}c^2\gamma\left( (\max_{i=1,\dots,n}\{\tilde\alpha_k^i\})^2 - (\max_{i=1,\dots,n}\{\tilde\alpha_{k-1}^i\})^2\right)
\]
and
\[
-\frac{1}{\theta^2}\max\{\tilde\alpha_k^i\}^2\geq  -\max\{\tilde\alpha_{k-1}^i\}^2
\]
we obtain
\[
\Phi_k - \Phi_{k-1} =  \frac{1}{2}c^2\gamma\left( \max_{i=1,\dots,n}\{\tilde\alpha_k^i\}^2 - \max_{i=1,\dots,n}\{\tilde\alpha_{k-1}^i\}^2\right) \leq 
\frac{1}{2}c^2\gamma\left( \max_{i=1,\dots,n}\{\tilde\alpha_k^i\}^2 -\frac{1}{\theta^2}\max\{\tilde\alpha_k^i\}^2\right),
\]
so that
\[
\Phi_k - \Phi_{k-1} \leq 
-\frac{1}{2}c^2\gamma\left(  \frac{1 - \theta^2}{\theta^2}\right) \max_{i=1,\dots,n}\{\tilde\alpha_k^i\}.
\]
Let us now consider the case when $(k-1)$-th iteration is of ``success'', i.e. $x_k\neq x_{k-1}$ so that $f(x_k) < f(x_{k-1})$. Then, we consider the following three cases:
\begin{enumerate}
	\item $\displaystyle\max_{i=1,\dots,n}\{\tilde\alpha_k^i\} = \displaystyle\max_{i=1,\dots,n}\{\tilde\alpha_{k-1}^i\}$;
	
	\item $\displaystyle\max_{i=1,\dots,n}\{\tilde\alpha_k^i\} > \displaystyle\max_{i=1,\dots,n}\{\tilde\alpha_{k-1}^i\}$;
	
	\item $\displaystyle\max_{i=1,\dots,n}\{\tilde\alpha_k^i\} < \displaystyle\max_{i=1,\dots,n}\{\tilde\alpha_{k-1}^i\}$.
	
\end{enumerate}

\underline{Case 1}. Since the $(k-1)$-th iteration is of success, there is an index $i$ such that, whichever version of LAM we consider, the following holds
\[
  f(y_{k-1}^{i+1}) \leq f(y_{k-1}^i) - \gamma(\bar\alpha_{k-1}^i)^2
\]
since in both versions of LAM the extrapolation cycle (in both the DF-Linesearch procedures) gets started when
\[
  f(y_{k-1}^i + \bar\alpha_{k-1}^i\hat d) \leq f(y_{k-1}^i) - \gamma(\bar\alpha_{k-1}^i)^2.
\]
Then, considering that
\[
  \bar\alpha_{k-1}^i = \max\{\tilde\alpha_{k-1}^i,c\max_{j=1,\dots,n}\{\tilde\alpha_{k-1}^j\}\} \geq c\max_{j=1,\dots,n}\{\tilde\alpha_{k-1}^j\},
\]
we have
\[
f(y_{k-1}^{i+1}) \leq f(y_{k-1}^i) - \gamma(\bar\alpha_{k-1}^i)^2 \leq f(y_{k-1}^i) - \gamma c^2\max_{j=1,\dots,n}\{\tilde\alpha_{k-1}^j\}^2.
\]
Then, recalling that we are in case 1, 
\[
   f(y_{k-1}^{i+1}) \leq f(y_{k-1}^i) - \gamma c^2\max_{j=1,\dots,n}\{\tilde\alpha_{k}^j\}^2.
\]
Moreover, since by definition $f(x_k)\leq f(y_{k-1}^{n+1})$ and $f(y_{k-1}^i) \leq f(x_{k-1})$, we can write
\[
f(x_k) \leq f(y_{k-1}^{i+1}) \leq f(y_{k-1}^i) - \gamma c^2\max_{j=1,\dots,n}\{\tilde\alpha_{k}^j\}^2 \leq f(x_{k-1}) - \gamma c^2\max_{j=1,\dots,n}\{\tilde\alpha_{k}^j\}^2.
\]
Then, we have
\[
\Phi_k - \Phi_{k-1} = f(x_k) - f(x_{k-1}) \leq - \gamma c^2\max_{i=1,\dots,n}\{\tilde\alpha_{k}^i\}^2.
\]

\underline{Case 2}. Since $\displaystyle\max_{i=1,\dots,n}\{\tilde\alpha_k^i\} > \displaystyle\max_{i=1,\dots,n}\{\tilde\alpha_{k-1}^i\}$, we have that an index $\bar\jmath$ exists such that
\[
\max_{i=1,\dots,n}\{\tilde\alpha_{k}^i\} = \tilde\alpha_{k}^{\bar\jmath}
\]
and a linesearch has been performed along the $\bar\jmath$-th direction, so that, whichever version of LAM we consider that
\[
  f(y^{\bar\jmath +1}_{k-1}) \leq f(y^{\bar\jmath}_{k-1}) -\gamma(\tilde\alpha_{k}^{\bar\jmath})^2.
\]
Hence, 
\[
f(x_k) - f(x_{k-1}) \leq -\gamma(\tilde\alpha_{k}^{\bar\jmath})^2,
\]
so that
\begin{eqnarray*}
\Phi_k - \Phi_{k-1} && = f(x_k) - f(x_{k-1}) +\frac{1}{2}c^2\gamma\left(\max_{i=1,\dots,n}\{\tilde\alpha_k^i\}^2 - \max_{i=1,\dots,n}\{\tilde\alpha_{k-1}^i\}^2\right) \\
  && \leq -\gamma \max_{i=1,\dots,n}\{\tilde\alpha_k^i\}^2 +\frac{1}{2}c^2\gamma\left(\max_{i=1,\dots,n}\{\tilde\alpha_k^i\}^2 - \max_{i=1,\dots,n}\{\tilde\alpha_{k-1}^i\}^2\right) \\
  && \leq -\gamma \max_{i=1,\dots,n}\{\tilde\alpha_k^i\}^2 +\frac{1}{2}c^2\gamma\max_{i=1,\dots,n}\{\tilde\alpha_k^i\}^2 \\
  && \leq -\gamma \left(1 - \frac{c^2}{2}\right)\max_{i=1,\dots,n}\{\tilde\alpha_k^i\}^2.
\end{eqnarray*}

\underline{Case 3}.
In this case, we know that an index $\bar\imath$ exists such that
\[
  \tilde\alpha_k^{\bar\imath} \geq \bar\alpha_{k-1}^{\bar\imath} = \max\{\tilde\alpha_{k-1}^{\bar\imath}, c\max_{i=1,\dots,n}\{\tilde\alpha_{k-1}^i\}\}\geq c\max_{i=1,\dots,n}\{\tilde\alpha_{k-1}^i\}.
\]
Then, we can write (recalling that the iteration is of success and whichever version of LAM we consider)
\[
f(x_k) \leq f(x_{k-1}) -\gamma c^2\max_{i=1,\dots,n}\{\tilde\alpha_{k-1}^i\}^2.
\]
Hence, we have
\begin{eqnarray*}
\Phi_k - \Phi_{k-1} &=& f(x_k) - f(x_{k-1}) + \gamma \frac{c^2}{2} \max_{i=1,\dots,n}\{\tilde\alpha_{k}^i\}^2 - \gamma \frac{c^2}{2} \max_{i=1,\dots,n}\{\tilde\alpha_{k-1}^i\}^2  \\
  &\leq& -\gamma c^2\max_{i=1,\dots,n}\{\tilde\alpha_{k-1}^i\}^2 + \gamma \frac{c^2}{2} \left(\max_{i=1,\dots,n}\{\tilde\alpha_{k}^i\}^2 - \max_{i=1,\dots,n}\{\tilde\alpha_{k-1}^i\}^2 \right) \\
  &=& - \gamma c^2 \max_{i=1,\dots,n}\{\tilde\alpha_{k}^i\}^2+  \gamma {c^2} \left(\max_{i=1,\dots,n}\{\tilde\alpha_{k}^i\}^2 - \max_{i=1,\dots,n}\{\tilde\alpha_{k-1}^i\}^2 \right) \\
  && \quad + \gamma \frac{c^2}{2} \left(\max_{i=1,\dots,n}\{\tilde\alpha_{k}^i\}^2 - \max_{i=1,\dots,n}\{\tilde\alpha_{k-1}^i\}^2 \right)\\
  &=& - \gamma c^2 \max_{i=1,\dots,n}\{\tilde\alpha_{k}^i\}^2+  
   \frac{3}{2}\gamma c^2 \left(\max_{i=1,\dots,n}\{\tilde\alpha_{k}^i\}^2 - \max_{i=1,\dots,n}\{\tilde\alpha_{k-1}^i\}^2 \right)\\
   &< &- \gamma c^2 \max_{i=1,\dots,n}\{\tilde\alpha_{k}^i\}^2,
\end{eqnarray*}
where the last inequality follows from the fact that we are in case 3.
Hence, recalling the  above three cases, for all $k$ we can always write
\begin{equation}\label{eq:bound_Phikkm1}
\Phi_{k} -\Phi_{k-1} \leq -\tilde c \max_{i=1,\dots,n}\{\tilde\alpha_{k}^i\}^2,
\end{equation}
where $\tilde c$ is defined in (\ref{ctildedef}).
Then, considering that 
\[
\Phi_{\bar\jmath+1} - \Phi_{0} =  (\Phi_{\bar\jmath+1} - \Phi_{\bar\jmath }) + (\Phi_{\bar\jmath } - \Phi_{\bar\jmath - 1}) + \dots + (\Phi_{1} - \Phi_0) 
\]
and recalling (\ref{eq:bound_Phikkm1}) we can write
\[
\Phi_{\bar\jmath+1} - \Phi_{0} \leq  - \tilde c \sum_{k=1}^{\bar\jmath+1}\max_{i=1,\dots,n}\{\tilde\alpha_k^i\}^2 = - \tilde c \sum_{k=0}^{\bar\jmath}\max_{i=1,\dots,n}\{\tilde\alpha_{k+1}^i\}^2.
\]
By recalling (\ref{eq:lowerbound_Phi}), we can write
\begin{equation}\label{eq:bound_diff_Phi}
  f_{\min} -\Phi_0 \leq \Phi_{\bar\jmath} - \Phi_{0}\leq - \tilde c \sum_{k=0}^{\bar\jmath}\max_{i=1,\dots,n}\{\tilde\alpha_{k+1}^i\}^2.
\end{equation}
Now, by (\ref{bound_maxalpha1}), we have 
\[
\max_{i=1,\dots,n}\{\tilde\alpha_{k+1}^i\}^2 \geq \frac{\eps^2}{c_1^2n},\quad\mbox{for}\ k=0,1,\dots,\bar\jmath,
\]
and, from (\ref{eq:bound_diff_Phi}), we can write
\[
 \Phi_0 - f_{\min} \geq \tilde c \sum_{k=0}^{\bar\jmath}\max_{i=1,\dots,n}\{\tilde\alpha_{k+1}^i\}^2 \geq (\bar\jmath+1) \tilde c\frac{\eps^2}{n c_1^2}.
\]
Thus, the number of iterations $\bar\jmath$ can be bounded from above by
\begin{equation*}
   \bar\jmath \leq \left\lceil\frac{n c_1^2 (\Phi_0 - f_{\min})}{\tilde c}\eps^{-2}\right\rceil= {\cal O}(\eps^{-2})
\end{equation*}
which concludes the proof.
\end{proof}

Now, we prove the worst case complexity bound for the number of function evaluations.

\begin{proposition}
Suppose that Assumption \ref{gradlip} holds and that $f$ is bounded from below. Given any $\epsilon\in(0,1)$, assume that $\bar\jmath+1$ is the first iteration such that $\|\nabla f(x_{\bar\jmath})\|\leq\epsilon$, i.e. 
 $\|\nabla f(x_k)\| > \epsilon$, for all $k=0,1,\dots,\bar\jmath$. Then, the number of function evaluations $Nf$ required by the LAM in the worst case is such that 
 \[
  Nf \leq {\cal O}(\epsilon^{-2}).
 \]
\end{proposition}
\begin{proof}
For all $k=0,1,\dots,\bar\jmath$, by Proposition \ref{bound_grad}, we have that
\begin{equation}\label{bound_maxalpha_success_app}
 \epsilon < \|\nabla f(x_k)\| \leq c_1 \sqrt{n}\max_{i=1,\dots,n}\{\tilde\alpha^i_{k+1}\},
 \end{equation}
where $c_1$ is defined in (\ref{c1def}).

\underline{For the standard LAM}. 
Let $U_{\bar\jmath}$ be the index set of unsuccessful iterations. For every $k\in U_{\bar\jmath}$, the algorithm performs 
\begin{equation}\label{nfku_app}
Nf_k^u = 2n
\end{equation}
function evaluations.

On the other hand, if $k\in S_{\bar\jmath}$ (i.e. $k$ is a successful iteration), we can distinguish the function evaluations performed by the algorithm in those producing a sufficient decrease in the objective function value and those producing a failure, i.e. the last function evaluation performed by the DF-Linesearch procedure. Hence, in this case we can further distinguish the function evaluations as $Nf_k^s$ and $\overline{Nf}_k^s$. As concerns $\overline{Nf}_k^s$ we have
\begin{equation}\label{nfksbar_app}
\overline{Nf}_k^s \leq 2n.
\end{equation}
For every $k\in S_{\bar{j}}$, we can bound the difference $f(x_k)-f(x_{k+1})$ as 
\begin{equation}\label{suff_dec_w_exp_app}
 f(x_{k}) - f(x_{k+1})  \geq \sum_{i \in \tilde{I}_k} \gamma (\delta^{-h_k^i} \bar{\alpha}_k^i)^2 \geq \sum_{i \in \tilde{I}_k} \gamma (\delta^{-h_k^i} c \max_{i = 1, \dots, n}\{\tilde{\alpha}_k^i\})^2,
\end{equation}
where $\tilde{I}_k$ is the set of indices where the sufficient decrease is found, hence where the ``DF-Linesearch" procedure is performed, and $h_k^i \in {\cal{Z_+}}\cup \{0\}$ is the number of expansions performed by the DF-Linesearch and producing sufficient decrease. Using (\ref{bound_maxalpha_success_app}), we can write
\begin{equation}\label{suff_dec_eps_app}
 f(x_{k}) - f(x_{k+1})  \geq \sum_{i \in \tilde{I}_k}\gamma \frac{c^2}{c_1^2n} (\delta^{-h_k^i} \epsilon)^2.
\end{equation}
Summing up over all the successful iterations we get
\begin{equation}\label{bound_exp_app}
 f(x_{0}) - f(x_{\bar{j}})  \geq \gamma \frac{c^2}{c_1^2n} \epsilon^2 \sum_{k \in S_{\bar{j}}} \sum_{i \in \tilde{I}_k} \delta^{-2h_k^i},
\end{equation}
so we can write
\begin{equation}\label{bound_success_app}
 \sum_{k \in S_{\bar{j}}} \sum_{i \in \tilde{I}_k} \delta^{-2h_k^i} \leq \frac{c_1^2 n\left(f(x_{0}) - f(x_{\bar{j}})\right)}{c^2 \gamma} \epsilon^{-2}.
\end{equation}
Now, let us consider the function of one variable
\[
   \phi(a) = (a+1)\delta^{2a}
\]
and the problem
\[
  \max_{a\geq 0}\phi(a).
\]
The derivative of $\phi$ is
\[
 \phi'(a) = \delta^{2a}\left( 1 + 2(a+1)\ln{\delta}\right).
\]
so that $\phi$ only has the stationary point at  
\[
  a^* = \frac{-2\ln{\delta}-1}{2\ln{\delta}},
\]
and, considering that $\delta\in(0,1)$,
\[
\phi(a^*) = -\frac{1}{2\ln{\delta}}\delta^{\left(-2 -\frac{1}{\ln{\delta}}\right)} > 0.
\]
Since $\phi(a)\to 0$ for $a\to +\infty$ and $\phi(a)\to-\infty$ for $a\to-\infty$,  $a^*$ is the global maximizer of $\phi$ on $\Re$. Hence, $\phi(a^*)\geq \phi(a)$ for all $a\in\Re$ and in particular $\phi(a^*)\geq \phi(0) = 1$. Then, we can write
\[
\phi(a) \leq \max\{1,\phi(a^*)\}=\varphi^*
\]
for $a\geq 0$. Then, for each $k\in S_{\bar\jmath}$ and for every $i\in \tilde I_k$, we have
\[
   h_k^i + 1 \leq \varphi^* \delta^{-2h_k^i},
\]
so that
\begin{equation}\label{h_plus_1_app}
\begin{split}
 Nf^s = \sum_{k\in S_{\bar\jmath}}Nf_k^s & = \sum_{k \in S_{\bar{j}}} \sum_{i \in \tilde{I}_k} (h_k^i + 1) \leq \varphi^*\sum_{k \in S_{\bar{j}}} \sum_{i \in \tilde{I}_k} \delta^{-2h_k^i} \\
 & \leq \varphi^*\frac{c_1^2 n\left(f(x_{0}) - f(x_{\bar{j}})\right)}{c^2 \gamma} \epsilon^{-2}.
\end{split}
\end{equation}
Then, recalling (\ref{nfku_app}), (\ref{nfksbar_app}) and (\ref{h_plus_1_app}), the total number of function evaluations can be bounded by
\[
Nf \leq 2n(|U_{\bar\jmath}|+|S_{\bar\jmath}|) + Nf^s \leq 2n \bar\jmath +\varphi^*\frac{c_1^2 n\left(f(x_{0}) - f(x_{\bar{j}})\right)}{c^2 \gamma} \epsilon^{-2}.
\]
Recalling (\ref{bound_on_iterations}), we finally get
\begin{equation}\label{Nf_standard}
Nf \leq 2n \left\lceil\frac{n c_1^2 (\Phi_0 - f_{\min})}{\tilde c}\eps^{-2}\right\rceil + \left\lceil \varphi^*\frac{nc_1^2 \left(f(x_{0}) - f(x_{\bar{j}})\right)}{c^2 \gamma} \epsilon^{-2}\right\rceil = {\cal O}(\epsilon^{-2}),
\end{equation}
where $c_1$ and $\tilde c$ are defined in (\ref{c1def}) and (\ref{ctildedef}), respectively.

\underline{For the new LAM}.
Then, for every iteration $k$, if $k\in U_{\bar\jmath}$, the Algorithm performs 
\[
Nf_k^u = 2n, 
\]
function evaluations. 

On the other hand, if $k\in S_{\bar\jmath}$, we can distinguish the function evaluations performed by the algorithm in those producing a sufficient decrease in the objective function value and those producing a failure, i.e. the last function evaluation performed by the DF-Linesearch procedure. Hence, in this case we can further distinguish the function evaluations as $Nf_k^s$ and $\overline{Nf}_k^s$. As concerns $\overline{Nf}_k^s$ we have
\[
\overline{Nf}_k^s \leq 2n.
\]
Concerning $Nf_k^s$, every time that one such function evaluation is performed, we have by the instructions of the DF-Linesearch procedure that
\[
\begin{split}
   f(y_k^i + \alpha_j\hat d) - f(y_k^i + \alpha_j/\delta \hat d) & \geq \gamma \left(\frac{1-\delta}{\delta}\right)^2\alpha_j^2\geq \gamma \left(\frac{1-\delta}{\delta}\right)^2 c^2\max_{i=1,\dots,n}\{\tilde\alpha_k^i\}^2 \\
   & \geq \gamma \left(\frac{1-\delta}{\delta}\right)^2 c^2\frac{\epsilon^2}{nc_1^2},
\end{split}
\]
where the last inequality follows from (\ref{bound_maxalpha_success_app}).
Then, recalling that $f$ is bounded from below by $f_{\min}$, summing the above relation over all such function evaluations, we obtain
\[
   f_0 - f_{\min} \geq Nf^s\gamma \left(\frac{1-\delta}{\delta}\right)^2 c^2\frac{\epsilon^2}{nc_1^2},
\]
so that
\[
Nf^s \leq \frac{c_1^2n(f_0 - f_{\min})}{\gamma c^2\epsilon^2}\frac{\delta^2}{(1-\delta)^2}.
\]

Finally, recalling the the number of iterations performed by the algorithm is bounded by ${\cal O}(\epsilon^{-2})$, and denoting by $Nf$ the total number of function evaluations performed by the algorithm, we can write in the worst case
\begin{equation}\label{Nf_new}
Nf \leq 2n\left\lceil\frac{n c_1^2 (\Phi_0 - f_{\min})}{\tilde c}\eps^{-2}\right\rceil + \left\lceil\frac{nc_1^2(f_0 - f_{\min})}{\gamma c^2\epsilon^2}\frac{\delta^2}{(1-\delta)^2}\right\rceil =  {\cal O}(\epsilon^{-2}),
\end{equation}
where $c_1$ and $\tilde c$ are defined in (\ref{c1def}) and (\ref{ctildedef}), respectively.
The proof is concluded recalling (\ref{Nf_standard}) and (\ref{Nf_new}).
\end{proof}

\section{Conclusions}\label{sec:conclusion}
In this paper we are concerned about the worst case complexity of linesearch-based derivative-free algorithm for the unconstrained optimization of a black-box objective function. In particular, we propose a general framework, namely  LAM, based on a suitable derivative-free linesearch procedure. We managed to show that the algorithm model takes at most ${\cal O}(n^2\epsilon^{-2})$ iterations and ${\cal O}(n^3\epsilon^{-2})$ function evaluations to drive the norm of the gradient below $\epsilon$. We note that, our complexity bounds are worse than those obtained in \cite{vicente:13} for direct search methods. In particular, in \cite{vicente:13} it has been shown that a direct search method with sufficient decrease achieves a norm of the gradient less than $\epsilon$ in at most ${\cal O}(n\epsilon^{-2})$ iterations and ${\cal O}(n^2\epsilon^{-2})$ function evaluations. However, it must be said that these differences in the complexity bounds are tied to the structure of the framework that we choose to analyze. In particular, in one iteration of  LAM new points $y_k^i$ are produced by taking possibly non-zero step along the search directions. This particular feature of LAM gives great freedom of movement and can positively impact the efficiency of the overall scheme. However, when it comes to bounding the norm of the gradient, this freedom is balanced by a $(\sqrt{n})^2$  rather than $\sqrt{n}$ in the coefficient used in the bound of the gradient.

We remark that it is possible to retain the same complexity bounds of \cite{vicente:13} by appropriately modifying the LAM framework. More in details, in the following we report a modified LAM that explores the search directions always starting from $x_k$ and than chooses the best point to define the new iterate. For this framework we can quite easily obtain the same complexity bounds of \cite{vicente:13} both for the number of iterations and the number of function evaluations but still using different step sizes along the search directions and the extrapolation procedure. In this way we preserve the same strong asymptotic convergence properties and the bound on the norm of the gradient in every iteration (not only in the failure ones).

\noindent\framebox[\textwidth]{\parbox{0.95\textwidth}{
\par\bigskip
\centerline{ {\bf  (modified) LAM }}
\par\medskip
  {\bf Data.} $c\in(0,1),\theta\in(0,1)$, $x_0\in \R^n$, $\tilde\alpha_0^i > 0$, $i\in \{1,\dots,n\}$, and set $d_0^i=e^i$, for $i=1,\ldots,n$.

\par\medskip

{\bf For} $k=0,1,\dots$


\qquad {\bf For} $i=1,\dots,n$

%

\qquad\qquad Set $y_k^i=x_k$.

\qquad\qquad Let $\bar\alpha_k^i = \max\{\tilde\alpha_k^i,c\displaystyle\max_{j=1,\dots,n}\{\tilde\alpha_k^j\}\}$.


\qquad\qquad Compute $\alpha$ and $d$ by the {\tt DF-Linesearch}$(\bar\alpha_k^i,y_k^i,d_k^i;\alpha,d)$.

\medskip

\qquad\qquad {\bf If} $\alpha = 0$ {\bf then} set $\alpha_k^i = 0$ and $\tilde\alpha_{k+1}^i = \theta\bar\alpha_k^i$.

\qquad\qquad {\bf else} set $\alpha_k^i = \alpha$, $\tilde\alpha_{k+1}^i = \alpha$.



\medskip

 \qquad\qquad Set $d_{k+1}^i=d$. 
\par\medskip

 \qquad {\bf End For}

 \qquad Set $x_{k+1}=\displaystyle\argmin_{i=1,\dots,n}\{f(x_k + \alpha_k^i d_{k+1}^i)$\}.

{\bf End For}

\par\bigskip\noindent

}}

However, it should be noted that the better complexity bounds of the modified LAM framework does not typically reflect in a more performing algorithm from a computational point of view.


\begin{thebibliography}{1}
	
	\bibitem{amaral2022complexity}
	V.~S. Amaral, R.~Andreani, E.~G. Birgin, D.~S. Marcondes, and J.~M.
	Mart{\'\i}nez.
	\newblock On complexity and convergence of high-order coordinate descent
	algorithms for smooth nonconvex box-constrained minimization.
	\newblock {\em Journal of Global Optimization}, pages 1--35, 2022.
	
	\bibitem{audet:06}
	C.~Audet and J.~E. {Dennis Jr.}
	\newblock Mesh adaptive direct search algorithms for constrained optimization.
	\newblock {\em SIAM Journal on Optimization}, 17(1):188--217, 2006.
	
	\bibitem{cartis2022evaluation}
	C.~Cartis, N.~I.~M. Gould, and Ph.~L. Toint.
	\newblock Evaluation complexity of algorithms for nonconvex optimization:
	Theory, computation and perspectives, 2022.
	
	\bibitem{Dodangeh:16}
	M.~Dodangeh, L.~N. Vicente, and Z.~Zhang.
	\newblock On the optimal order of worst case complexity of direct search.
	\newblock {\em Optimization Letters}, 10(4):699–708, 2016.
	
	\bibitem{fasano2014linesearch}
	G.~Fasano, G.~Liuzzi, S.~Lucidi, and F.~Rinaldi.
	\newblock A linesearch-based derivative-free approach for nonsmooth constrained
	optimization.
	\newblock {\em SIAM journal on optimization}, 24(3):959--992, 2014.
	
	\bibitem{grippo1988global}
	L.~Grippo, F.~Lampariello, and S.~Lucidi.
	\newblock Global convergence and stabilization of unconstrained minimization
	methods without derivatives.
	\newblock {\em Journal of Optimization Theory and Applications},
	56(3):385--406, 1988.
	
	\bibitem{lucidi2002global}
	S.~Lucidi and M.~Sciandrone.
	\newblock On the global convergence of derivative-free methods for
	unconstrained optimization.
	\newblock {\em SIAM Journal on Optimization}, 13(1):97--116, 2002.
	
	\bibitem{torczon:97}
	V.~Torczon.
	\newblock On the convergence of pattern search algorithms.
	\newblock {\em SIAM Journal on Optimization}, 7(1):1--25, 1997.
	
	\bibitem{vicente:13}
	L.~N. Vicente.
	\newblock Worst case complexity of direct search.
	\newblock {\em EURO Journal on Computational Optimization}, 1:143--153, 2013.
	
\end{thebibliography}

\end{document}